\documentclass{amsart}
\usepackage{graphicx}
\usepackage{amsmath,amssymb}

\newtheorem{theorem}{Theorem}[section]

\newtheorem{assertion}[theorem]{Assertion}
\newtheorem{corollary}[theorem]{Corollary}

\theoremstyle{definition}

\theoremstyle{remark}
\newtheorem{remark}[theorem]{Remark}

\numberwithin{equation}{section}

\begin{document}

\title{Power convergence of Abel averages}%
\author{Yuri Kozitsky}%
\address{Institute of Mathematics, Maria Curie-Sklodowska University, 20-031 Lublin, Poland}%
\email{jkozi@hektor.umcs.lublin.pl}%
\author{David Shoikhet}
\address{Department of Mathematics, ORT Braude College, P.O. Box 78, 21982 Karmiel,
Israel}
\email{davs@braude.ac.il}
\author{Jaroslav Zem\'anek}
\address{Institute of Mathematics, Polish Academy of Sciences, P.O.Box 21, 00-956 Warsaw,
Poland} 
\email{zemanek@impan.pl}

\begin{abstract}
Necessary and sufficient conditions are presented for the Abel
averages of discrete and strongly continuous semigroups, $T^k$ and
$T_t$, to be power convergent in the operator norm in a complex
Banach space. These results cover also the case where $T$ is
unbounded and the corresponding Abel average is defined by means of
the resolvent of $T$. They complement the classical results by
Michael Lin establishing sufficient conditions for the corresponding
convergence for a bounded $T$.

\end{abstract}
\maketitle

\section{Posing the problem}

For a bounded linear operator $T$ on a Banach space $X$, the Abel
average of the discrete semigroup $\{T^k\}_{k\in \mathbb{N}_0}$ is
defined as
\begin{equation}
  \label{e1.1}
  A_\alpha = (1 -\alpha)
\sum_{ k=0}^\infty \alpha^k T^k = (1 - \alpha)[I -\alpha ®T]^{-1},
\end{equation}
where $\alpha$ is a suitable numerical parameter, i.e., such that
$A_\alpha$ belongs to $\mathcal{L}(X)$ -- the Banach algebra of all
bounded linear operators\footnote{Actually, the operators $A_\alpha$
have a natural geometric origin, even in a more general nonlinear
setting, see \cite[page 154]{18} and \cite{19}.} on $X$.

Likewise, for a strongly continuous semigroup $\{T_t\}_{t\geq0}$,
the Abel average is defined by the formula
\begin{equation}
  \label{e1.2}
\tilde{A}_\lambda = \lambda \int_0^\infty e^{-\lambda s} T_s ds,
\end{equation}
with a suitable parameter $\lambda$, which is to be understood
point-wise, as an improper Riemann integral; see, e.g., \cite[page
42]{4}.

In this note, we establish necessary and sufficient conditions which
ensure that the averages (\ref{e1.1}) and (\ref{e1.2}) are power
convergent in the operator norm. Our main result (Theorem \ref{t2.1}
below) covers also the case where $T$ in (1.1) is unbounded.

The study of the Abel averages goes back to at least E. Hille
\cite{6} and W.F. Eberlein \cite{3}. They are presented in the books
\cite{4,7,10,17}. Uniform ergodic theorems for Abel and Ces{\`a}ro
averages were established by M. Lin in \cite{12} and \cite{13}. The
following assertions can be deduced from the corresponding nowadays
classical results of \cite{13}.
\begin{assertion}
  \label{a1.1}
Let $T$ be such that
\begin{equation}
  \label{e1.3}
\|T^n /n\| \to 0, \qquad { as} \qquad n\to \infty.
\end{equation}
Then, for each $\alpha \in (0,1)$, the operator $A_\alpha$ in
(\ref{e1.1}) belongs to $\mathcal{L}(X)$, and the following
statements are equivalent: \vskip.2cm
\begin{tabular}{ll}
(i) \ &$(I-T)X$ is closed;\\[.2cm]
(ii) \ &the net $\{A_\alpha\}_{\alpha \in (0,1)}$ converges in
$\mathcal{L}(X)$, as $\alpha\to 1^{-}$;\\[.2cm]
(iii) \ &the Ces{\`a}ro averages $N^{-1} \sum_{n=0}^{N-1} T^n$
converge, as $N\to \infty$.
\end{tabular}
\vskip.2cm \noindent The (operator-norm) limit in (ii) and (iii) is
the same -- the projection $E$ of $X$ onto ${\rm Ker} (I-T)$ along
${\rm Im} (I-T)$, that is, the Riesz projection corresponding to the
(at most) simple pole $1$ of the resolvent of $T$.
\end{assertion}
\begin{assertion}
  \label{a1.2}
 Let $\{T_t\}_{t\geq 0}$ be a strongly continuous semigroup of bounded linear
operators such that
\begin{equation}
  \label{e1.4}
\|T_t /t\| \to 0, \qquad { as} \qquad t\to + \infty,
\end{equation}
and let $B$ be its generator. Then, for all $\lambda > 0$, the
operator $\tilde{ A}_\lambda$ in (\ref{e1.2}) is in $\mathcal{L}(X)$
and the following statements are equivalent: \vskip.2cm
\begin{tabular}{ll}
(i) \ &$B$ has closed range;\\[.2cm]
(ii) \ &the net $\{\tilde{A}_\lambda\}_{\lambda >0}$ converges in
$\mathcal{L}(X)$, as $\lambda\to 0^{+}$;\\[.2cm]
(iii) \ &for each $\lambda>0$, the operator $\tilde{A}_\lambda$ is
uniformly ergodic, that is,\\ &the sequence of its Ces{\`a}ro
averages $N^{-1} \sum_{n=0}^{N-1} \tilde{A}_\lambda^n$ converges\\ &in
$\mathcal{L}(X)$.
\end{tabular}
\vskip.2cm \noindent The limits in (ii) and (iii) coincide; their
common value is the projection  $\tilde{E}$ of $X$ onto ${\rm Ker}B$
along ${\rm Im}B$, given by the Riesz decomposition\footnote{See,
e.g., \cite[Theorem 18.8.1, pages 521--522]{7} and \cite[Theorems
5.8-A and 5.8-D, pages 306--311]{21}.}
\begin{equation}
  \label{e1.5}
X = {\rm Ker}B \oplus {\rm Im}B,
\end{equation}
corresponding to the (at most) simple pole $0$ of the resolvent of
$B$.
\end{assertion}
Note that
\[
{\rm Ker} B = \bigcap_{t\geq 0} {\rm Ker}(I-T_t),
\]
where the inclusion $``\subset"$ follows by, e.g., \cite[Theorem
1.8.3, page 33]{17}.

In the discrete case, an analog of claim (iii) of Assertion
\ref{a1.2} can also be obtained. As follows from (\ref{e1.3}), the
spectrum of $T$ is contained in the closure of the open unit disk
$\varDelta$. By the spectral mapping theorem, the spectrum of
$A_\alpha$ is then contained in $\varDelta \cup \{1\}$. Since
\[
{\rm Ker} (I-A_\alpha) = {\rm Ker}(I-T) \qquad {\rm and} \qquad {\rm
Im} (I-A_\alpha) = {\rm Im}(I-T),
\]
cf. the proof of Theorem \ref{t2.1} below, we have the Riesz
decomposition
\[
{\rm Ker} (I-A_\alpha) \oplus {\rm Im} (I-A_\alpha) = {\rm Ker}
(I-T) \oplus {\rm Im} (I-T) =X,
\]
and thus the point $1$ is at most a simple pole of $A_\alpha$. In
particular, it is at most an isolated point of the spectrum of
$A_\alpha$. Hence, $\|A_\alpha^n/n\|\to 0$, as $n \to  +\infty$;
see, e.g., \cite{15}. Therefore, all the operators $A_\alpha$,
$\alpha \in(0, 1)$, are uniformly ergodic, even power convergent to
the same limit $E$ as above. This complements Assertion \ref{a1.1}
in the spirit of Assertion \ref{a1.2}.

As we shall see in Assertions \ref{a1.3} and \ref{a1.4} below, both
claims (ii) above are equivalent to the power convergence of the
corresponding Abel averages; see also Remark \ref{r2.2} below.
Indeed, under the conditions of Assertions \ref{a1.1} and
\ref{a1.2}, by the technique used in \cite{13} one can show that,
for $\alpha$ close to $1^{-}$ and $\lambda$ close to
$0^+$, the operators $A_\alpha$ and $\tilde{A}_\lambda$,
respectively, are power convergent in $\mathcal{L}(X)$. As we shall
see later, if $X$ is a complex Banach space, the assumptions of
Assertions \ref{a1.1} and \ref{a1.2} allow one to prove the
corresponding power convergence of the operators $A_\alpha$ and
$\tilde{A}_\lambda$, for all $\alpha\in (0,1)$ and all $\lambda >
0$, respectively. More precisely, the following extensions of
Assertions \ref{a1.1} and \ref{a1.2} hold. See also \cite{14}.
\begin{assertion}
  \label{a1.3}
Let $T$ be a bounded linear operator in a complex Banach space $X$
obeying (\ref{e1.3}), and let $A_\alpha$, $\alpha \in (0,1)$, be its
Abel average (\ref{e1.1}). Then the following statements are
equivalent: \vskip.2cm
\begin{tabular}{ll}
(i)  &$(I-T)X$ is closed;\\[.2cm]
(ii)  &for some $\alpha\in (0,1)$,  the sequence
$\{A^n_\alpha\}_{n\in \mathbb{N}}$ converges in
$\mathcal{L}(X)$;\\[.2cm]
(iii)  &for each $\alpha\in (0,1)$,  the sequence
$\{A^n_\alpha\}_{n\in \mathbb{N}}$ converges in $\mathcal{L}(X)$.
\end{tabular}
\vskip.2cm \noindent The limits in (ii) and (iii) coincide with the
projection $E$ from Assertion \ref{a1.1}.
\end{assertion}
\begin{assertion}
  \label{a1.4}
Let $\{T_t\}_{t\geq 0}\subset \mathcal{L}(X)$ be a strongly
continuous semigroup of bounded linear operators in a complex Banach
space $X$ such that (\ref{e1.4}) holds. Let $B$ be its generator and
$\tilde{A}_\lambda$, $\lambda >0$, be its Abel average (\ref{e1.2}).
Then the following statements are equivalent: \vskip.2cm
\begin{tabular}{ll}
(i) \ &$B$ has closed range;\\[.2cm]
(ii) \ &for some $\lambda >0$, the sequence
$\{\tilde{A}^n_\lambda\}_{n\in \mathbb{N}}$ converges in
$\mathcal{L}(X)$;\\[.2cm]
(iii) \ &for each $\lambda>0$, the sequence
$\{\tilde{A}^n_\lambda\}_{n\in \mathbb{N}}$ converges in
$\mathcal{L}(X)$.
\end{tabular}
\vskip.2cm \noindent The limits in (ii) and (iii) coincide with the
projection  $\tilde{E}$ from Assertion \ref{a1.2}.
\end{assertion}
In fact, the conditions (\ref{e1.3}) and (\ref{e1.4}) are quite far
from being necessary for the corresponding Abel averages to converge
as stated above. For example, the former one can be replaced by the
dissipativity condition used in the classical Lumer--Phillips
theorem; see, e.g., \cite[page 250]{22}. The next assertion, which
provides an example of this sort, might be useful in the study of
the sets of fixed points of some nonlinear operators; see \cite{18}
and \cite{19}.
\begin{assertion}
\label{a1.5} For a complex Banach space $X$ and
$T\in\mathcal{L}(X)$, let $W(T)$ denote the numerical range of
$T$; see \cite[page 81]{BD} or \cite[page 12]{17}, and let
$\overline{W(T)}$ be its closure. Suppose that $T$ is such that
${\rm Re}W(T)\subset (-\infty ,1]$. Then, for each $\alpha \in (0, 1)$, the Abel
averages (\ref{e1.1}) of $T$ obey the estimate $\|A_\alpha\|\leq  1$
and statements (i), (ii), and (iii) of Assertion \ref{a1.3} are
equivalent. Furthermore, if ${\rm Re}\overline{W(T)} \subset (-\infty ,1)$, then $I-
T$ is invertible on $X$ and $\lim_{n\to +\infty}A^n_\alpha = 0$.
\end{assertion}
Also, the assumption (1.3) in Assertion \ref{a1.3} can be relaxed to
\begin{equation}
  \label{e1.6}
  \sup_{N} \bigg{\|}\frac{1}{N} \sum_{n=0}^{N-1} T^n \bigg{\|} <
  \infty,
\end{equation}
which, by \cite[Theorem 3.1]{16}, is equivalent to
\begin{equation}
  \label{e1.7}
\sup_{\alpha\in (0, 1)}  \sup_{N \in\mathbb{N}_0} \bigg{\|}(1-\alpha)\sum_{k=0}^N \alpha^k T^k \bigg{\|}  <
\infty.
\end{equation}
Indeed, by \cite{5}, condition (\ref{e1.6})
and the closedness of $(I-T)X$ yield the existence of
$\lim_{\alpha\to 1^{-}} A_\alpha$, which is equivalent to the fact
that the point $1$ is at most a simple pole of the resolvent of $T$;
see \cite[Theorem 18.8.1, pages 521--522]{7}. Hence, by the
Koliha--Li characterization of the power convergence \cite{8,9,11},
statements (i), (ii), and (iii) of Assertion \ref{a1.3} are again
equivalent, this time under the weaker assumption (\ref{e1.6}) in
place of (\ref{e1.3}).

In the light of the above facts, it would be interesting to find an
analogous characterization of the norm-boundedness in $t > 0$ of the
integral averages
\[
\frac{1}{t} \int_0^t T_s ds,
\]
assuming, e.g., the uniform boundedness of the partial integrals in
(\ref{e1.2}).

Of course, the uniform Abel boundedness (\ref{e1.7}) is by no means
necessary for the existence of $\lim_{\alpha\to 1^{-}} A_\alpha$.
Relevant matrix examples can easily be constructed by using
\cite[Theorem 8, page 378]{24}.

\section{The results}

In this section, we derive the conditions that are necessary and
sufficient for the statements of Assertions \ref{a1.3}, \ref{a1.4},
and \ref{a1.5} to hold. Moreover, our results cover also the case
where $T$ in (\ref{e1.1}) is unbounded, and hence (\ref{e1.3}) is
not applicable. The key observation which allowed us to get them is
that the principal thing one needs is the spectrum $\sigma(T)$ lying
merely in the half-plane
\[
\Pi = \{\zeta \in \mathbb{C}: {\rm Re}  \zeta \leq 1\}.
\]
Note also that (\ref{e1.3}) and statement (i) in Assertion
\ref{a1.3} imply that
\begin{equation}
  \label{e2.1}
{\rm Ker}(I-T)\oplus {\rm Im}(I-T) = X;
\end{equation}
see, e.g., \cite{15} and \cite[pages 40--43]{18}. In the sequel, for
a closed densely defined linear operator $T$ in a complex Banach
space $X$, by $\mathcal{D}(T)$ and $\rho(T)$ we denote the domain
and the resolvent set of $T$, respectively. For such an operator
with $(1,+\infty) \subset \rho(T)$, the Abel average can be defined
as the following bounded linear operator
\begin{equation}
  \label{e2.2}
  A_\alpha = (1-\alpha)[I-\alpha T]^{-1}, \qquad \alpha \in (0,1).
\end{equation}
Finally, by ${\rm Im}(I - T)$ we mean the $(I- T)$-image of
$\mathcal{D}(T)$.
\begin{theorem}
  \label{t2.1}
Let T be a densely defined closed linear operator in a complex
Banach space $X$ such that $(1,+\infty) \subset  \rho(T)$. Then the following statements are equivalent: \vskip.2cm
\begin{tabular}{ll}
(i) \quad & for each $\alpha \in (0,1)$, the sequence
$\{A_\alpha^n\}_{n\in \mathbb{N}}$ of powers of its\\ &Abel average (\ref{e2.2}) converges in $\mathcal{L}(X)$;\\[.2cm]
(ii) \quad &$\sigma (T) \subset \Pi$ and (\ref{e2.1}) holds.
\end{tabular}
\vskip.2cm \noindent For every $\alpha \in (0,1)$, the limit in (i)
is the projection of $X$ onto ${\rm Ker} (I-T)$ along ${\rm Im}
(I-T)$.
\end{theorem}
\begin{proof}
For $\lambda \in \rho(T)$, let $R(\lambda ,T)$ denote the resolvent
of $T$. Thus, we have
\begin{eqnarray*}
(a) \quad & &(\lambda I-T)R(\lambda, T) x = x, \qquad x\in X;\\[.2cm]
(b) \quad & &R(\lambda, T)(\lambda I-T) x = x, \qquad x\in
\mathcal{D}(T).
\end{eqnarray*}
For $\alpha \in (0,1)$, we have $1/\alpha \in\rho(T)$;
hence, $A_\alpha = (\alpha^{-1} -1)R(\alpha^{-1}, T)$,
which yields
\begin{eqnarray}
\label{e2.3}
(a) \quad & &( I-\alpha T)A_\alpha x = (1-\alpha)x, \qquad x\in X;\\[.2cm]
(b) \quad & &A_\alpha ( I- \alpha T) x = (1-\alpha) x, \qquad x\in
\mathcal{D}(T). \nonumber
\end{eqnarray}
Take now an $x\in {\rm Ker} (I-A_\alpha)$,  that is, $A_\alpha
x = x$. As ${\rm Im} A_\alpha$ lies in $\mathcal{D}(T)$, our $x$ is
in $\mathcal{D}(T)$, and by (a) in (\ref{e2.3}) we have that
$x-\alpha T x = x - \alpha x$. Thus, $x = T x$, and hence ${\rm
Ker}(I - A_\alpha) \subset {\rm Ker}(I - T)$. Conversely, choose $x
\in {\rm Ker}(I- T) \subset\mathcal{ D}(T)$. Then by (b) in
(\ref{e2.3}), we have that $A_\alpha(x - \alpha x) = (1-\alpha )x$.
Hence, $x = A_\alpha x$, and thereby
\begin{equation}
  \label{e2.4}
  {\rm Ker} (I-A_\alpha) = {\rm Ker}(I-T).
\end{equation}
Let now $x$ be in ${\rm Im}(I -T)$, that is, $x = y -Ty$ for some $y
\in \mathcal{D}(T)$. By (b) in (\ref{e2.3}), we then get
$\alpha (I-T) y = (I-A_\alpha) (I-\alpha T)y$,
which yields $x= (I-A_\alpha)z$ for
\[
z = \frac{1}{\alpha} (I-\alpha T) y.
\]
Therefore, ${\rm Im}(I -T) \subset {\rm Im}(I-A_\alpha)$.
Conversely, let $x \in {\rm Im}(I-A_\alpha)$, i.e., $x = y -
A_\alpha y$, for some $y\in X$. Note that $z = \alpha(1 -
\alpha)^{-1} A_\alpha y$ is in $\mathcal{D}(T)$. For this $z$, by
(a) in (\ref{e2.3}) we have
\begin{eqnarray*}
  (I-T)z  =  \frac{\alpha}{1-\alpha} (I-T) A_\alpha y
   =  y - A_\alpha y = (I - A_\alpha )y = x,
\end{eqnarray*}
which finally yields
\begin{equation}
  \label{e2.5}
{\rm Im}(I- A_\alpha ) = {\rm Im} (I-T).
\end{equation}
Let us stress that both (\ref{e2.4}) and (\ref{e2.5}) hold for any
$\alpha\in (0, 1)$. Moreover, the subspaces in (\ref{e2.4}) and
(\ref{e2.5}) are closed whenever $A_\alpha$ is power convergent.

For an $\alpha \in (0, 1)$, consider the following univalent
analytic function
\begin{equation}
  \label{e2.6}
 f_\alpha (\zeta) = \frac{1-\alpha}{1 - \alpha \zeta}, \qquad \zeta
 \in \mathbb{C}\setminus \{ \alpha^{-1}\}.
\end{equation}
It maps the domain
$\Omega_\alpha = \{ \zeta \in \mathbb{C}: |\alpha \zeta - 1| > 1 -
\alpha\}$
onto the open unit disk $\varDelta \subset \mathbb{C}$, and
$f_\alpha(1) = 1$. Obviously, $A_\alpha = f_\alpha(T)$, and
$\sigma(A_\alpha)$ lies in the closure of $\varDelta$ (actually, it
lies in $\varDelta\cup \{1\}$ by the Koliha--Li characterization of
the power convergence). Thus, by the spectral mapping theorem (see,
e.g., \cite[Theorem 5.71-A, page 302]{21}) and our assumption
$(1,+\infty)\subset  \rho(T)$, we obtain that $\sigma(T)$ lies in
$\overline{\Omega}_\alpha$ -- the closure of $\Omega_\alpha$. Therefore,
\[
\sigma (T) \subset \bigcap_{\alpha \in (0,1)} \overline{\Omega}_\alpha = \Pi.
\]
Moreover, (\ref{e2.4}) and (\ref{e2.5}) yield (\ref{e2.1}), by the Koliha-Li
characterization of the power convergence \cite{8,9,11}. Thus, $(i) \Rightarrow  (ii)$.

For each $\alpha\in (0, 1)$, the
homographic transformation (\ref{e2.6}) maps $\Pi$ onto the closed
disk $\{\zeta\in \mathbb{C}: |\zeta - 1/2|\leq 1/2$; see, e.g.,
\cite[page 84]{20}. This yields $\sigma(A_\alpha) \subset \varDelta
\cup\{1\}$. Since ${\rm Ker}(I- A_\alpha) = {\rm Ker}(I - T)$ and
${\rm Im}(I- A_\alpha) = {\rm Im}(I - T)$, it follows by
(\ref{e2.1}) that, for each $\alpha \in (0, 1)$, the powers
$A^n_\alpha$ converge to the projection $E$ of $X$ onto ${\rm Ker}(I
- T)$ along ${\rm Im}(I - T)$, where $E$ is as in Assertion
\ref{a1.1}.
\end{proof}
\begin{remark}
  \label{r2.2}
Condition (\ref{e2.1}) in (ii) of Theorem \ref{t2.1} can be replaced
by the existence of $\lim_{\alpha \to 1^{-}}A_\alpha$. In view of
(\ref{e2.4}) and (\ref{e2.5}), the latter limit is equal to the
Riesz projection $E$ of $X$ onto ${\rm Ker}(I- T)$ along ${\rm
Im}(I- T)$, given by the decomposition in (\ref{e2.1}). The point
$1$ is simultaneously at most a simple pole of the resolvents of
both $T$ and $A_\alpha$.
\end{remark}

The theorem just proven obviously extends Assertion \ref{a1.3}.
Since the closure of the numerical range of a bounded linear
operator contains its spectrum (see, e.g., \cite[Theorem 10.1, page
88]{BD}, \cite[Theorem 1.3.9, page 12]{17}, or \cite[Proposition,
page 217]{23}), Theorem \ref{t2.1} is also a generalization of
Assertion \ref{a1.5}. In the same spirit, we obtain the following
extension of Assertion \ref{a1.4}.
\begin{theorem}
  \label{t2.3}
Let $\{T_t\}_{t\geq 0}$ be a strongly continuous semigroup in a
complex Banach space $X$. Let $B$ be its generator and $\tilde{
A}_\lambda$ be its Abel average (\ref{e1.2}). Additionally, assume
that $\rho(B)$ contains the positive real axis. Then the following
statements are equivalent: \vskip.2cm
\begin{tabular}{ll}
(i) \quad &$\rho(B)$ contains the whole open right half-plane and\\
&${\rm Ker} B \oplus {\rm Im}B = X$;\\[.2cm]
(ii) \quad &for some $\lambda >0$, the sequence $\{
\tilde{A}_\lambda^n\}_{n\in \mathbb{N}}$ converges in
$\mathcal{L}(X)$\\
& and $\rho(B)$ contains the whole open right half-plane;\\[.2cm]
(iii) \quad &for each $\lambda >0$, the sequence $\{
\tilde{A}_\lambda^n\}_{n\in \mathbb{N}}$ converges in
$\mathcal{L}(X)$.
\end{tabular}
\vskip.2cm \noindent For each $\lambda>0$, the limit above is the
projection of $X$ onto ${\rm Ker} B$ along ${\rm Im}B$.
\end{theorem}
\begin{proof}
For $\lambda >0$, we have $(1+\lambda)^{-1} =: \alpha \in (0,1)$.
Set $T= I+B$. Then
\begin{eqnarray*}
\tilde{A}_\lambda & = & \lambda (\lambda I -B)^{-1} =
\frac{1-\alpha}{\alpha} \left(\frac{1-\alpha}{\alpha} I - (T-I)
\right)^{-1}\\[.2cm] & = & (1-\alpha)[(1-\alpha) I - \alpha
(T-I)]^{-1} \\[.2cm]
& = & (1-\alpha) [I-\alpha T]^{-1} = A_\alpha.
\end{eqnarray*}
Now the proof follows directly from Theorem \ref{t2.1}.
\end{proof}

With the help of \cite[Theorem VIII.1.11, page 622]{2} we get the
following generalization of Assertion \ref{a1.4}. Recall that the
Abel average $\tilde{ A}_\lambda$ was defined in (\ref{e1.2}) and
its $n$-th power can be written as (see, e.g., \cite[page 43]{4})
\begin{equation}
  \label{e2.7}
\tilde{A}_\lambda^n = \lambda^n [R(\lambda, A)]^n =
\frac{\lambda^n}{(n-1)!} \int_0^\infty e^{-\lambda t} t^{n-1} T_t
dt, \quad n=1, 2, \dots .
\end{equation}
As mentioned just after Assertion \ref{a1.2}, we have
\[
{\rm Ker} B = \bigcap_{t\geq 0} \{x\in X: T_t x = x\}.
\]
\begin{corollary}
  \label{c2.4}
Let $\{T_t\}_{t\geq 0}$ be a strongly continuous semigroup in a
complex Banach space $X$. Let $B$ be its generator and $\tilde{
A}_\lambda$ be its Abel average (\ref{e1.2}). Assume also that
\begin{equation}
  \label{e2.8}
\lim_{t\to +\infty} \frac{\log \|T_t\|}{t} = 0.
\end{equation}
Then the following statements are equivalent: \vskip.2cm
\begin{tabular}{ll}
(i) \quad &${\rm Ker} B \oplus {\rm Im} B = X$;\\[.2cm]
(ii) \quad &for some $\lambda >0$,  the sequence
$\{\tilde{A}_\lambda^n \}_{n\in \mathbb{N}}$ converges in
$\mathcal{L}(X)$;\\[.2cm]
(iii) \quad &for each $\lambda >0$, the sequence
$\{\tilde{A}_\lambda^n \}_{n\in \mathbb{N}}$ converges in
$\mathcal{L}(X)$.
\end{tabular}
\vskip.2cm \noindent For each $\lambda>0$, the limit in (ii) and
(iii) is the projection of $X$ onto ${\rm Ker} B$ along ${\rm Im}
B$.
\end{corollary}
\begin{remark}
  \label{r2.5}
In Theorem \ref{t2.3} and Corollary \ref{c2.4}, the convergence in
claims (ii) and (iii) is based either on the condition $(0,+\infty)
\subset \rho(B)$ or on (\ref{e2.8}). Both are weaker than
(\ref{e1.4}) used in Assertion \ref{a1.4}; see, e.g., \cite[Theorem
VIII.1.11, page 622]{2} or \cite[Corollary II.1.3, page 44]{4}.
Like in Remark
\ref{r2.2}, either (ii) or (iii) can be replaced by the existence
of $\lim_{\lambda \to 0^+} \tilde{A}_\lambda$, due to \cite[Theorem 18.8.1, pages 521--522]{7}.
\end{remark}

\section{An example}

We present an unbounded linear operator $T$, which has the
properties described by Theorem \ref{t2.1}. Here $X$ is the complex
Hilbert space $L^2(\mathbb{R})$.
\begin{equation}
  \label{e3.1}
T_0 = D^2 + (2 - t^2), \qquad D= \frac{d}{dt}, \qquad
\mathcal{D}(T_0)=S(\mathbb{R}),
\end{equation}
where
$S(\mathbb{R})$ is the space of Schwartz test functions.
Then $T_0$ is essentially self-adjoint and such that
\begin{equation}
  \label{e3.2}
T_0 x_n = \lambda_n x_n, \qquad \lambda_n = 1 - 2 n, \quad n\in
\mathbb{N}_0.
\end{equation}
The eigenvalues $\lambda_n$ are simple and the eigenvectors
\begin{equation}
  \label{e3.3}
x_n (t) = h_n (t) \exp(- t^2/2), \qquad t\in \mathbb{R},
\end{equation}
constitute an orthonormal basis of $X$; see, e.g., \cite[pages
36--39]{1}. In (\ref{e3.3}), for $n\in \mathbb{N}_0$, $h_n$ is the
 Hermite polynomial of degree $n$.
In particular, $h_0 = \pi^{1/4}$. Let $T$ be the closure of
(\ref{e3.1}). Then $X_0:= {\rm Ker} (I-T)$ is the one-dimensional
subspace of $X$ spanned by $x_0$. Let $X_1$ be the orthogonal
complement of $X_0$, i.e.,
\begin{equation}
  \label{e3.4}
 X = X_0 \oplus X_1.
\end{equation}
Take any $x\in X_1$. Then
\begin{equation}
  \label{e3.5}
x= \sum_{n=1}^\infty \alpha_n x_n,
\end{equation}
and hence $x = (I- T)y$ for
\[
y = \sum_{n=1}^\infty \frac{\alpha_n}{2n} x_n.
\]
This immediately yields that $X_1 = {\rm Im}(I- T)$, and hence
\[
X = {\rm Ker}(I-T) \oplus {\rm Im}(I-T),
\]
by (\ref{e3.4}). For the resolvent of $T$, we have
\[
R(\lambda, T)x_0 = \frac{1}{\lambda -1}x_0, \qquad R(\lambda, T)x_n
= \frac{1}{\lambda - \lambda_n}x_n, \quad n\in \mathbb{N}.
\]
Thus, in view of (\ref{e3.2}), $R(\lambda, T)$ is a compact operator,
positive for $\lambda > 1$. Then its spectral
decomposition is
\begin{equation}
  \label{e3.6}
R(\lambda, T) = \sum_{n=0}^\infty \frac{1}{\lambda - \lambda_n} P_n,
\end{equation}
where $P_n$, $n\in \mathbb{N}_0$, is the orthogonal projection onto
the subspace spanned by $x_n$. For $\lambda > 1$ and any $x\in X$,
cf. (\ref{e3.5}), we have
\begin{eqnarray*}
\|(\lambda -1)R(\lambda, T)x - P_0 x \| & = & (\lambda -1) \left[
\sum_{n=1}^\infty \frac{|\alpha_n|^2}{(\lambda - 1
+2n)^2}\right]^{1/2}\\[.2cm]
& \leq & (\lambda -1) \|x\|,
\end{eqnarray*}
which yields that, in $\mathcal{L}(X)$, $(\lambda -1) R(\lambda, T) \to P_0$ as $\lambda
\to 1^{+}$.
For $m\in \mathbb{N}$, by (\ref{e3.6}) we have
\[
[(\lambda-1) R(\lambda,T)]^m = \sum_{n=0}^\infty \left(\frac{\lambda
-1}{ \lambda -1 + 2n} \right)^m P_n.
\]
Then, for $m \geq 4$,
\begin{eqnarray*}
\|[(\lambda -1) R(\lambda, T)]^m - P_0 \| & = &
\bigg{\|}\sum_{n=1}^\infty \left( \frac{\lambda -1}{\lambda - 1 +
2n}\right)^m P_n\bigg{\|}\qquad \\[.2cm]
& \leq & \left(\frac{\lambda-1}{\lambda +1}
\right)^{m-2}\sum_{n=1}^\infty \left( \frac{\lambda -1}{\lambda - 1
+ 2n}\right)^2 \\[.2cm]
& = & \left(\frac{\lambda-1}{\lambda +1} \right)^{m-2} C(\lambda)
\to 0, \quad {\rm as} \quad m\to +\infty.
\end{eqnarray*}

\subsection*{Acknowledgment}
This work was supported by the DFG through SFB 701 ``Spektrale
Strukturen und Topologische Methoden in der Mathematik" and through
the research project 436 POL 113/125/0-1, and also by the European
Commission Project TODEQ (MTKD-CT-2005-030042).

\end{document}